\documentclass{amsart}
\usepackage{amssymb}
\usepackage{graphicx}

\theoremstyle{plain}
\newtheorem{theorem}{Theorem}[section]
\newtheorem{corollary}[theorem]{Corollary}

\newtheorem{proposition}[theorem]{Proposition}
\newtheorem{lemma}[theorem]{Lemma}
\numberwithin{equation}{section}

\begin{document}

\title[Chebyshev-like polynomial sequences]{Properties of two Chebyshev-like polynomial sequences}

\author{Karl Dilcher}
\address{Department of Mathematics and Statistics\\
         Dalhousie University\\
         Halifax, Nova Scotia, B3H 4R2, Canada}
\email{dilcher@mathstat.dal.ca}
\author{Seon-Hong Kim}
\address{Department of Mathematics and Research Institute of Natural Science\\
         Sookmyung Women’s University\\
         Seoul, 140-742 Korea}
\email{\tt shkim17@sookmyung.ac.kr}
\author{Kenneth B. Stolarsky}
\address{Department of Mathematics\\
         University of Illinois\\
         1409 West Green Street, Urbana, IL 61801, USA}
\email{\tt stolarsk@math.uiuc.edu}
\thanks{Research supported in part by the Natural Sciences and Engineering
        Research Council of Canada, Grant \# 145628481}

\date{}

\setcounter{equation}{0}

\begin{abstract}
By modifying the generating function of the Chebyshev polynomials of the 
second kind, we obtain a sequence of reciprocal (or palindromic) polynomials,
as well as a related companion sequence. Among numerous other properties, we 
obtain discriminant and resultant identities for these polynomial sequences and
prove partial irreducibility results. Throughout, we point to parallels and
connections with the Chebyshev polynomials of both kinds.
\end{abstract}

\maketitle

\section{Introduction}\label{sec:1}

Pascal's triangle and Chebyshev polynomials of both kinds belong to the most 
important objects in the theory of special functions and related areas of
mathematics. We begin with the Chebyshev polynomials of the second kind,
$U_n(x)$, which among other equivalent definitions can be defined by the
generating function
\begin{equation}\label{1.1}
\frac{1}{1-2xt+t^2} = \sum_{n=0}^\infty U_n(x)t^n.
\end{equation}
The entries of Pascal's triangle as well, when seen as coefficients of the 
polynomial
$(x+1)^n$, can be given by a generating function through an easy geometric
series. For reasons that will soon become clear, we are only interested in the
even-index rows for Pascal's triangle, and we consider
\begin{equation}\label{1.2}
F(z,t):=\frac{1}{1-(z+1)^2t} = \sum_{n=0}^\infty(z+1)^{2n}t^n.
\end{equation}
If we now form the product $F(z,t)F(-z,t)$, then the Cauchy product of the
right-hand side of \eqref{1.2} and its ``$-z$" analogue gives
\begin{equation}\label{1.3}
F(z,t)F(-z,t)=\sum_{n=0}^\infty\left((z-1)^{2n}
\sum_{k=0}^n\left(\frac{z+1}{z-1}\right)^{2k}\right)t^n.
\end{equation}
On the other hand, the product of the generating functions gives
\begin{align}
F(z,t)F(-z,t)
&=\frac{1}{1-2\tfrac{z^2+1}{z^2-1}\left((z^2-1)t\right)+\left((z^2-1)t\right)^2}\label{1.4} \\
&=\sum_{n=0}^\infty(z^2-1)^nU_n(\tfrac{z^2+1}{z^2-1})t^n,\nonumber
\end{align}
where we have used \eqref{1.1}. Equating coefficients of $t^n$ in \eqref{1.3}
and \eqref{1.4}, we obtain after some straightforward manipulations,
\begin{align*}
U_n(\tfrac{z^2+1}{z^2-1}) &=\left(\frac{z-1}{z+1}\right)^n
\sum_{k=0}^n\left(\frac{z+1}{z-1}\right)^{2k} \\
&= \frac{z^2-1}{4z}\left(\left(\frac{z+1}{z-1}\right)^{n+1}
-\left(\frac{z-1}{z+1}\right)^{n+1}\right).
\end{align*}
This last identity could also be obtained from the well-known formula
\begin{equation}\label{1.6}
U_n(x)=\frac{\left(x+\sqrt{x^2-1}\right)^{n+1}
-\left(x-\sqrt{x^2-1}\right)^{n+1}}{2\sqrt{x^2-1}};
\end{equation}
see, e.g., \cite[p.~10]{Ri}.

We now go one step further and consider the generating function
\begin{equation}\label{1.7}
F(z,t)^2F(-z,t)=\sum_{n=0}^\infty p_n(z)t^n.
\end{equation}
As is the case with the sequences generated by \eqref{1.2} and \eqref{1.4},
the functions $p_n(z)$, $n=0, 1, \ldots,$ are polynomials; see 
Section~\ref{sec:2}. The first few of them are listed in Table~1.

\medskip
\begin{center}
{\renewcommand{\arraystretch}{1.1}
\begin{tabular}{|r|l|}
\hline
$n$ & $p_n(z)$ \\
\hline
0 & 1 \\
1 & $3z^2+2z+3$ \\
2 & $6z^4+8z^3 + 20z^2 + 8z + 6$ \\
3 & $10z^6 + 20z^5 + 70z^4 + 56z^3 + 70z^2 + 20z + 10$  \\
4 & $15z^8 + 40z^7 + 180z^6 + 216z^5 + 378z^4 + 216z^3 + 180z^2 + 40z + 15$  \\
5 & $21 z^{10}+70 z^9+385 z^8+616 z^7+1386 z^6+1188 z^5+1386 z^4+\cdots$  \\
\hline
\end{tabular}}

\medskip
{\bf Table~1}: $p_n(z)$ for $0\leq n\leq 5$.
\end{center}

\medskip
These polynomials will be the main objects of study in this paper.
In Section~\ref{sec:2} we derive a number of basic properties of the polynomials
$p_n(z)$, some of which will be needed in later sections. In Section~\ref{sec:3}
we establish connections between these polynomials and the Chebyshev
polynomials of both kinds. A companion sequence to the polynomials $p_n(z)$ is
introduced and investigated in Section~\ref{sec:4}, and in Section~\ref{sec:5}
we evaluate certain resultants and discriminants. Section~\ref{sec:6} is 
devoted to a partial irreducibility result for the polynomials $p_n(x)$.
We conclude this paper with some factorization results of certain related
sequences.

\section{Basic properties of the $p_n(z)$}\label{sec:2}

In this section we derive some basic properties of the polynomials $p_n(x)$
defined by \eqref{1.7}. We begin by taking the Cauchy product of $F(z,t)$
with itself, obtaining
\begin{equation}\label{2.1}
F(z,t)^2=\sum_{n=0}^\infty\left(\sum_{k=0}^n(z+1)^{2k}(z+1)^{2n-2k}\right)t^n
=\sum_{n=0}^\infty(n+1)(z+1)^{2n}t^n.
\end{equation}
The series on the right-hand side of \eqref{2.1} can also be obtained by 
multiplying both sides of \eqref{1.2} by $t$ and then differentiating with
respect to $t$.

Next, taking the Cauchy product of the ``$-z$" analogue of 
\eqref{1.2} with \eqref{2.1}, we get
\begin{equation}\label{2.2}
F(z,t)^2F(-z,t)=\sum_{n=0}^\infty
\left(\sum_{k=0}^n(k+1)(z+1)^{2k}(z-1)^{2n-2k}\right)t^n,
\end{equation}
and equating coefficients of $t^n$ in \eqref{2.2} and \eqref{1.7} gives the
explicit expression
\begin{equation}\label{2.3}
p_n(z) = \sum_{k=0}^n(k+1)(z+1)^{2k}(z-1)^{2n-2k},
\end{equation}
or equivalently,
\begin{equation}\label{2.4}
p_n(z) = (z-1)^{2n}\sum_{k=0}^n(k+1)\left(\frac{z+1}{z-1}\right)^{2k}.
\end{equation}
This shows that $p_n(z)$ is a polynomial of degree $2n$ with integer 
coefficients, and furthermore we see that
\begin{equation}\label{2.5}
z^{2n}p_n(\tfrac{1}{z}) = p_n(z), 
\end{equation}
which means that $p_n(z)$ is a {\it reciprocal} (or self-inversive or 
palindromic) polynomial.
Next, if we substitute $z=1$, respectively $z=-1$, in \eqref{2.3}, we 
immediately get the evaluations
\begin{equation}\label{2.6}
p_n(1) = (n+1)2^{2n},\qquad p_n(-1) = 2^{2n}.
\end{equation}
This means, in particular, that the sum of the coefficients of $p_n(z)$ is
$(n+1)2^{2n}$. 

The identity \eqref{2.3} is also the basis for deriving several different
recurrence relations:

\begin{proposition}\label{prop:2.1}
For $n\geq 2$ we have
\begin{align}
p_{n+1}(z) &= (z-1)^2p_n(z) + (n+2)(z+1)^{2n+2},\label{2.7}\\
(n+1)p_{n+1}(z) &= \left((2n+3)z^2+2z+2n+3\right)p_n(z) \label{2.8}\\
&\qquad -(n+2)(z^2-1)^2p_{n-1}(z),\nonumber\\
p_{n+1}(z) &= (3z^2+2z+3)p_n(z)-(3z^2-2z+3)(z+1)^2p_{n-1}(z)\label{2.9}\\
&\qquad+(z-1)^2(z+1)^4p_{n-2}(z),\nonumber
\end{align}
with initial conditions given in Table~1. The recurrences \eqref{2.7} and 
\eqref{2.8} are also valid for $n=0$ and $n=1$, respectively.
\end{proposition}

The recurrence relation \eqref{2.7} is non-homogeneous, while \eqref{2.8} is
homogeneous of order 2. However, this last one has non-constant coefficients in
terms of $n$. It is also of ``Somos type", in the sense that it is not clear
from the recurrence itself that the polynomial would have integer coefficients.

\begin{proof}[Proof of Proposition~\ref{prop:2.1}]
We multiply both sides of \eqref{2.3} by $(z-1)^2$ and add $(n+2)(z+1)^{2n+2}$.
Using \eqref{2.3} again, we get $p_{n+1}(z)$, which proves \eqref{2.7}.

Next we consider \eqref{2.7} and its analogue with $n$ replaced by $n-1$.
If we multiply the former by $n+1$ and the latter by $(n+2)(z+1)^2$, then the
non-homogeneous term can be eliminated, which leads to \eqref{2.8}.

Finally, to obtain \eqref{2.9}, we use the definition of $F(z,t)$ in \eqref{1.2}
to write the generating function \eqref{1.7} explicitly as
\begin{equation}\label{2.9a}
\frac{1}{1-(3z^2+2z+3)t+(3z^2-2z+3)(z+1)^2t^2-(z-1)^2(z+1)^4t^3}
=\sum_{n=0}^\infty p_n(z)t^n.
\end{equation}
Following a standard method for obtaining recurrence relations, we multiply
both sides of \eqref{2.9a} by the denominator of the left-hand side and 
collect like powers of $t$. Writing the denominator in \eqref{2.9a} for 
simplicity as $1-a_1t+a_2t^2-a_3t^3$ and suppressing the arguments of the
polynomials $p_n(z)$, we then get
\begin{align*}
1 = p_0+\left(p_1-a_1p_0\right)t&+\left(p_2-a_1p_1+a_2p_0\right)t^2\\
&+\sum_{n=3}^\infty\left(p_n-a_1p_{n-1}+a_2p_{n-2}-a_3p_{n-3}\right)t^n.
\end{align*}
By the uniqueness of power series, the coefficients of $t^n$ for $n=0,1,2$ give
the first three entries in Table~1, while for $n\geq 3$ we immediately 
get \eqref{2.9}.
\end{proof}

A recurrence relation such as \eqref{2.9} or the corresponding generating
function can typically be used to obtain explicit formulas, similar in
nature to \eqref{1.6} for the Chebyshev polynomials. Here we get the 
following expression.

\begin{proposition}\label{prop:2.2}
For all $n\geq 0$, 
\begin{equation}\label{2.10}
p_n(z)=\frac{(z-1)^{2n+4}-\left(z^2-(4n+6)z+1\right)(z+1)^{2n+2}}{16z^2}.
\end{equation}
\end{proposition}

\begin{proof}
The generating function in \eqref{1.7} has the partial fraction expansion
\begin{align}
F(z,t)^2F(-z,t)&=\frac{(z-1)^4}{16z^2}\cdot\frac{1}{1-(z-1)^2t}
-\frac{(z^2-1)^2}{16z^2}\cdot\frac{1}{1-(z+1)^2t}\label{2.11} \\
&\qquad+\frac{(z+1)^2}{4z}\cdot\frac{1}{(1-(z+1)^2t)^2},\nonumber
\end{align}
which was obtained by computer algebra and is easy to verify. Upon equating
coefficients of $t^n$ in \eqref{1.7} and \eqref{2.11} we then get 
\eqref{2.10} after some easy manipulations.
\end{proof}

As a first application of \eqref{2.10} we obtain explicit expressions for the
coefficients of $p_n(z)$. In what follows, we use the notation
\begin{equation}\label{2.12}
p_n(z) = \sum_{k=0}^{2n}a_k^{(n)} z^k.
\end{equation}

\begin{proposition}\label{prop:2.3}
For all $n\geq 0$ we have
\begin{equation}\label{2.13}
a_{2j}^{(n)}=\frac{n+2}{4}\binom{2n+2}{2j+1}\qquad(0\leq j\leq n)
\end{equation}
and
\begin{equation}\label{2.14}
a_{2j+1}^{(n)}=\frac{(n+2)(n-j)}{2(2j+3)}\binom{2n+2}{2j+1}
\qquad(0\leq j\leq n-1).
\end{equation}
In particular, $a_0^{(n)}=a_{2n}^{(n)}=(n+1)(n+2)/2$.
\end{proposition}

\begin{proof}
We rewrite \eqref{2.10} as
\begin{equation}\label{2.15}
p_n(z)=\frac{(z-1)^{2n+4}-(z+1)^{2n+4}+4(n+2)z(z+1)^{2n+2}}{16z^2}.
\end{equation}
Using binomial expansions, we see that the numerator of \eqref{2.15} is
\begin{equation}\label{2.16}
-2\sum_{j=-1}^n\binom{2n+4}{2j+3}z^{2j+3}
+4(n+2)\sum_{k=-1}^{2n+1}\binom{2n+2}{k+1}z^{k+2},
\end{equation}
where we have slightly shifted the summation indices. For even powers of $z$,
only the second sum is relevant, and with \eqref{2.15} and \eqref{2.12} we
immediately get \eqref{2.13}.

For odd powers of $z$, we first note that the coefficient of $z$ vanishes in
\eqref{2.16}, as it must. Furthermore, \eqref{2.15} and \eqref{2.16}, along
with \eqref{2.12}, give 
\begin{equation}\label{2.17}
a_{2j+1}^{(n)}
=\frac{1}{8}\left(-\binom{2n+4}{2j+3}+(2n+4)\binom{2n+2}{2j+2}\right).
\end{equation}
It is easy to verify that
\[
\binom{2n+4}{2j+3}=\frac{(n+2)(2n+3)}{(2j+3)(j+1)}\binom{2n+2}{2j+1}
\]
and
\[
(2n+4)\binom{2n+2}{2j+2}=\frac{(n+2)(2n-2j+1)}{j+1}\binom{2n+2}{2j+1}.
\]
This, together with \eqref{2.17}, gives \eqref{2.14} after another
straightforward manipulation.
\end{proof}

To supplement the various linear identities we have derived so far, we are
now going to prove the following quadratic identity, which can also be seen
as a recurrence relation.
 
\begin{proposition}\label{prop:2.4}
For all $n\geq 1$ we have
\begin{equation}\label{2.18}
p_n(z)^2-p_{n-1}(z)p_{n+1}(z) = (z+1)^{2n}p_n(-z).
\end{equation}
As a consequence, Tur\'an's inequality
\begin{equation}\label{2.19}
p_n(z)^2-p_{n-1}(z)p_{n+1}(z) > 0
\end{equation}
holds for all $z\in\mathbb{R}$, except for $z=-1$ where the expression on the
left vanishes.
\end{proposition}

This should be compared with similar identities for the Chebyshev polynomials
$T_n(x)$ and $U_n(x)$, namely
\[
T_n(z)^2-T_{n-1}(z)T_{n+1}(z) = 1-z^2,\qquad
U_n(z)^2-U_{n-1}(z)U_{n+1}(z) = 1.
\]
In the case of the Fibonacci numbers $F_n$, the identity
$F_n^2-F_{n-1}F_{n+1}=(-1)^{n+1}$ is known as Cassini's identity.

\begin{proof}[Proof of Proposition~\ref{prop:2.4}]
This is best done by using Proposition~\ref{prop:2.2}: Substitute \eqref{2.10}
into the left-hand side of \eqref{2.18}. Straightforward manipulations then
show that there are numerous cancellations, and the remaining terms can be
reduced to the right-hand side of \eqref{2.10} again, with $z$ replaced by $-z$.

For the second statement we note that, by \eqref{2.3}, $p_n(z)$ is strictly
positive for all real $z$. Therefore, the right-hand side of \eqref{2.18},
and thus the left-hand side of \eqref{2.19}, are strictly positive for all
real $z\neq -1$.
\end{proof}

\section{Connections with Chebyshev polynomials}\label{sec:3}

Chebyshev polynomials were already mentioned in this paper, and in particular
the polynomials $U_n(x)$ appeared in the Introduction. In this section,
we establish more direct connections between the Chebyshev polynomials of
both types and the polynomials $p_n(z)$. We begin with an easy consequence of
some identities from the Introduction.

\begin{proposition}\label{prop:3.1}
For all $n\geq 0$ we have
\begin{equation}\label{3.1}
p_n(z) = (z+1)^{2n}\sum_{k=0}^n\left(\frac{z^2-1}{(z+1)^2}\right)^k
U_k(\tfrac{z^2+1}{z^2-1}).
\end{equation}
\end{proposition}

\begin{proof}
Taking the Cauchy product of the series \eqref{1.2} and \eqref{1.4}, we get
\[
F(z,t)^2F(-z,t)=\sum_{n=0}^\infty\left(\sum_{k=0}^n(z^2-1)^k
U_k(\tfrac{z^2+1}{z^2-1})(z+1)^{2n-2k}\right)t^n,
\]
and upon equating coefficients of $t^n$ with \eqref{1.7}, we obtain \eqref{3.1}.
\end{proof}

The following identity is of a somewhat different type.

\begin{proposition}\label{prop:3.2}
For all $n\geq 0$ we have
\begin{equation}\label{3.2}
p_{n+1}(z)=(z+1)^2p_n(z)+\frac{1}{2}\left(\sqrt{1-z^2}\right)^{2n+3}
U_{2n+3}(\tfrac{1}{\sqrt{1-z^2}}).
\end{equation}
\end{proposition}

\begin{proof}
We rewrite \eqref{2.15} as
\begin{equation}\label{3.3}
8zp_n(z)=2(n+2)(z+1)^{2n+2}-\frac{(1+z)^{2n+4}-(1-z)^{2n+4}}{2z}
\end{equation}
and compare it with the slightly rewritten identity \eqref{1.6}, namely
\begin{equation}\label{3.4}
U_n(x)=x^n\frac{\left(1+\sqrt{1-x^{-2}}\right)^{n+1}
-\left(1-\sqrt{1-x^{-2}}\right)^{n+1}}{2\sqrt{1-x^{-2}}}.
\end{equation}
With $z=\sqrt{1-x^{-2}}$, the identities \eqref{3.3} and \eqref{3.4} then give
\begin{equation}\label{3.5}
8\sqrt{1-x^{-2}}p_n(\sqrt{1-x^{-2}})=2(n+2)(1+\sqrt{1-x^{-2}})^{2n+2}
-\frac{U_{2n+3}(x)}{x^{2n+3}}.
\end{equation}
Dividing both sides by 2, this can be rewritten in terms of $z$ as
\begin{equation}\label{3.6}
4zp_n(z)=(n+2)(z+1)^{2n+2}
-\frac{1}{2}\left(\sqrt{1-z^2}\right)^{2n+3}
U_{2n+3}(\tfrac{1}{\sqrt{1-z^2}}).
\end{equation}
Finally, using \eqref{2.7} to replace $(n+2)(z+1)^{2n+2}$, we get the desired
identity \eqref{3.2}
\end{proof}

Of course, \eqref{3.5} and \eqref{3.6} could also be considered as identities
in their own right, connecting $p_n(z)$ and the Chebyshev polynomial
$U_{2n+3}(x)$.

We obtain identities that are more symmetric by combining $p_n(z)$ and 
$p_n(-z)$ as follows. We begin with a lemma.

\begin{lemma}\label{lem:3.3}
For all $n\geq 0$ we have
\begin{align}
p_n(z)+p_n(-z)&=(n+2)\frac{(1+z)^{2n+2}-(1-z)^{2n+2}}{4z},\label{3.7}\\
p_n(z)-p_n(-z)&=\frac{n}{4z}\left((1+z)^{2n+2}+(1-z)^{2n+2}\right)\label{3.8}\\
&\qquad -\frac{(1-z^2)^2}{8z^2}\left((1+z)^{2n}-(1-z)^{2n}\right). \nonumber
\end{align}
\end{lemma}

\begin{proof}
In \eqref{3.3}, we replace $z$ by $-z$, obtaining
\begin{equation}\label{3.9}
-8zp_n(-z)=2(n+2)(1-z)^{2n+2}-\frac{(1+z)^{2n+4}-(1-z)^{2n+4}}{2z}.
\end{equation}
Subtracting \eqref{3.9} from \eqref{3.3} and dividing by $8z$, we immediately
get \eqref{3.7}.

To obtain \eqref{3.8}, we first rewrite
\[
(1\pm z)^{2n+4}=\left(1\pm 4z+6z^2\pm 4z^3+z^4\right)(1\pm z)^{2n},
\]
which leads to 
\begin{align}
(1+z)^{2n+4}-(1-z)^{2n+4}
&=\left(1-z^2\right)^2\left((1+z)^{2n}-(1-z)^{2n}\right)\label{3.10}\\
&\qquad +4z\left((1+z)^{2n+2}+(1-z)^{2n+2}\right). \nonumber
\end{align}
Finally, adding \eqref{3.9} and \eqref{3.3} and then using \eqref{3.10}, we
obtain \eqref{3.8} after some easy manipulations.
\end{proof}

To apply Lemma~\ref{lem:3.3}, we need the explicit formula \eqref{1.6} for the
Chebyshev polynomials $U_n(x)$ and its analogue
\begin{equation}\label{3.11}
T_n(x)=\frac{1}{2}\left(\left(x+\sqrt{x^2-1}\right)^n
+\left(x-\sqrt{x^2-1}\right)^n\right).
\end{equation}
The $T_n(x)$ are the Chebyshev polynomials of the first kind, which can be 
defined by the generating function
\begin{equation}\label{3.12}
\frac{1-xt}{1-2xt+t^2} = \sum_{n=0}^\infty T_n(x)t^n;
\end{equation}
see, e.g., \cite[p.~41]{Ri} or \cite[(18.12.8)]{DLMF}.

\begin{proposition}\label{prop:3.4}
For all $n\geq 0$ we have
\begin{align}
p_n(\tfrac{\sqrt{x^2-1}}{x})+p_n(-\tfrac{\sqrt{x^2-1}}{x})
&=\frac{(n+2)U_{2n+1}(x)}{2x^{2n+1}},\label{3.13}\\
p_n(\tfrac{\sqrt{x^2-1}}{x})-p_n(-\tfrac{\sqrt{x^2-1}}{x})
&=\frac{2nxT_{2n+2}(x)-U_{2n-1}(x)}{4x^{2n+2}\sqrt{x^2-1}}. \label{3.14}
\end{align}
\end{proposition}

\begin{proof}
This is similar to the proof of Proposition~\ref{prop:3.2}. We set
$z=\sqrt{1-x^{-2}}=\sqrt{x^2-1}/x$ in \eqref{3.7} and use \eqref{3.4}, which
immediately gives \eqref{3.13}. 

Similarly, we rewrite \eqref{3.11} as
\begin{equation}\label{3.15}
T_n(x)=\frac{x^n}{2}\left(\left(1+\sqrt{1-x^{-2}}\right)^n
+\left(1-\sqrt{1-x^{-2}}\right)^n\right).
\end{equation}
Setting $z=\sqrt{x^2-1}/x$ in \eqref{3.8} and using \eqref{3.4} and 
\eqref{3.15}, we get \eqref{3.14} after some simplifications.
\end{proof}

The identity \eqref{3.13} will be used later, in Section~\ref{sec:6}.
As special cases of \eqref{3.13} and \eqref{3.14}, we obtain identities 
involving Fibonacci numbers $F_n$ and Lucas numbers $L_n$. As is well-known,
the Fibonacci numbers are defined recursively by $F_0=0$, $F_1=1$, and 
$F_{n+1}=F_n+F_{n-1}$ $(n\geq 1)$, while the Lucas numbers satisfy the same
recurrence, but with initial values $L_0=2$ and $L_1=1$.

\begin{corollary}\label{cor:3.5}
For all $n\geq 0$ we have
\begin{align}
p_n(\sqrt{5})+p_n(-\sqrt{5})
&= (n+2)4^nF_{2n+2},\label{3.16} \\
p_n(\sqrt{5})-p_n(-\sqrt{5})
&=4^n\frac{nL_{2n+2}-2F_{2n}}{\sqrt{5}}. \label{3.17}
\end{align}
\end{corollary}

\begin{proof}
It is well-known that 
\begin{equation}\label{3.18}
F_{n+1}=(-i)^nU_n(i/2),\qquad L_n=2(-i)^n T_n(i/2);
\end{equation}
see, e.g., \cite[pp. 61f]{Ri}. On the other hand, with $x=i/2$, we get
$\sqrt{x^2-1}/x=\sqrt{5}$. If we substitute this and the identities in 
\eqref{3.18} into \eqref{3.13} and \eqref{3.14}, we immediately get \eqref{3.16}
and \eqref{3.17}, respectively.
\end{proof}

\section{A companion sequence}\label{sec:4}

One may ask what happens if in \eqref{3.1} we replace the Chebyshev polynomials
of the second kind with those of the first kind, that is, if we define the
polynomials 
\begin{equation}\label{4.1}
q_n(z) := (z+1)^{2n}\sum_{k=0}^n\left(\frac{z^2-1}{(z+1)^2}\right)^k
T_k(\tfrac{z^2+1}{z^2-1}).
\end{equation}
We will see that this sequence also has interesting properties and can be 
considered a companion sequence of the polynomials $p_n(z)$. The first few
polynomials $q_n(z)$ are listed in Table~2.

\medskip
\begin{center}
{\renewcommand{\arraystretch}{1.1}
\begin{tabular}{|r|l|}
\hline
$n$ & $q_n(z)$ \\
\hline
0 & 1 \\
1 & $2z^2+2z+2$ \\
2 & $3z^4+6z^3+14z^2+6z+3$ \\
3 & $4z^6+12z^5+44z^4+40z^3+44z^2+12z+4$  \\
4 & $5z^8+20z^7+100z^6+140z^5+238z^4+140z^3+100z^2+20z+5$  \\
5 & $6z^{10}+30z^9+190z^8+360z^7+828z^6+756z^5+828z^4+\cdots$  \\
\hline
\end{tabular}}

\medskip
{\bf Table~2}: $q_n(z)$ for $0\leq n\leq 5$.
\end{center}

\medskip
Considering the proof of Proposition~\ref{prop:3.1}, it is clear that 
\eqref{4.1} also comes from a similar Cauchy product. With appropriate 
substitutions in \eqref{3.12} (see also \eqref{1.4}), we get the generating
function
\begin{equation}\label{4.2}
\frac{1-(z^2+1)t}{(1-(z+1)^2t)^2(1-(z-1)^2t)}=\sum_{n=0}^\infty q_n(z)t^n.
\end{equation}
We could now use standard methods to obtain numerous properties of the
polynomials $q_n(z)$. The following properties will be used later.

\begin{proposition}\label{prop:4.1}
For all $n\geq 1$ we have
\begin{equation}\label{4.3}
q_n(z)=p_n(z)-(z^2+1)p_{n-1}(z)
\end{equation}
and
\begin{equation}\label{4.3a}
q_n(z)+2zp_{n-1}(z)=(n+1)(z+1)^{2n}.
\end{equation}
\end{proposition}

It is interesting to compare these identities with the well-known relations
\[
T_n(z) = U_n(z) - z U_{n-1}(z),\quad
T_n(z) + U_{n-1}(z)\sqrt{z^2-1} = \big(z+\sqrt{z^2-1}\big)^n,
\]
which connect both kinds of Chebyshev polynomials.

\begin{proof}[Proof of Proposition~\ref{prop:4.1}]
Comparing the generating functions \eqref{4.2} and \eqref{1.7}, we get
\begin{align*}
\sum_{n=0}^\infty q_n(z)t^n
&=\left(1-(z^2+1)t\right)\cdot\sum_{n=0}^\infty p_n(z)t^n \\
&=\sum_{n=0}^\infty p_n(z)t^n - \sum_{n=0}^\infty(z^2+1)p_n(z)t^{n+1} \\
&=1 + \sum_{n=1}^\infty\left(p_n(z)-(z^2+1)p_{n-1}(z)\right)t^n.
\end{align*}
Upon equating coefficients of $t^n$, we get \eqref{4.3} for $n\geq 1$.

To prove \eqref{4.3a}, we rewrite \eqref{2.7} as 
\[
(n+1)(z+1)^{2n} = p_n(z)-(z-1)^2p_{n-1}(z)
\]
and subtract this identity from \eqref{4.3}, obtaining
\[
q_n(z)-(n+1)(z+1)^{2n} = -2zp_{n-1}(z).
\]
This completes the proof.
\end{proof}

Some basic properties of the polynomials $q_n(z)$ are similar to those of
$p_n(z)$. The proofs of the statements in the next corollary are quite obvious
and follow from \eqref{4.3} and the corresponding properties of $p_n(z)$.

\begin{corollary}\label{cor:4.2}
For all $n\geq 1$,\\
$(a)$\; $q_n(z)$ is a reciprocal polynomial of degree $2n$ with integer
coefficients;\\
$(b)$\; we have
\[
q_n(-1)=2^{2n-1},\qquad q_n(0)=n+1,\qquad q_n(1)=(n+2)2^{2n-1}.
\]
\end{corollary}

The integer sequence given by $q_n(1)$, i.e., the sum of the coefficients, has
an interesting combinatorial interpretation: According to entry A292044 in
\cite{OEIS}, it is the Wiener index of the $n$-halved cube graph.

From among the various possible recurrences we only mention the nonhomogeneous
relation
\[
q_{n+1}(z)=(z-1)^2q_n(z)+\left(z^2+(2n+4)z+1\right)(z+1)^{2n},
\]
valid for all $n\geq 0$, with $q_0(z)=1$; it is easily obtained by substituting
\eqref{2.7} into \eqref{4.3}. Other recurrences can be obtained as outlined in
the proofs of \eqref{2.8} and \eqref{2.9}. In particular, since the generating
functions of $p_n(z)$ and $q_n(z)$ have the same denominators, the recurrence
\eqref{2.9} will also be satisfied by the sequence $q_n(z)$, but with 
different initial conditions.

In analogy to Proposition~\ref{prop:2.3} we now derive explicit expressions
for the coefficients of $q_n(z)$, using the notation

\begin{equation}\label{4.6}
q_n(z) = \sum_{k=0}^{2n}b_k^{(n)} z^k.
\end{equation}

\begin{proposition}\label{prop:4.3}
For all $n\geq 0$ we have
\begin{equation}\label{4.7}
b_{2j}^{(n)}=\frac{2nj+2n-2j^2+1}{4n+2}\binom{2n+2}{2j+1}
\qquad(0\leq j\leq n)
\end{equation}
and
\begin{equation}\label{4.8}
b_{2j+1}^{(n)}=\frac{n+1}{2}\binom{2n}{2j+1}
\qquad(0\leq j\leq n-1).
\end{equation}
In particular, $b_0^{(n)}=b_{2n}^{(n)}=n+1$.
\end{proposition}

\begin{proof}
Combining \eqref{4.3} with \eqref{2.12}, we get 
\begin{align*}
q_n(z)&=\sum_{k=0}^{2n}a_k^{(n)}z^k-\sum_{k=0}^{2n-2}a_k^{(n-1)}z^k
-\sum_{k=0}^{2n-2}a_k^{(n-1)}z^{k+2} \\
&= \sum_{k=0}^{2n}\left(a_k^{(n)}-a_k^{(n-1)}-a_{k-2}^{(n-1)}\right)z^k.
\end{align*}
Comparing this with \eqref{4.6}, we get
\begin{equation}\label{4.9}
b_k^{(n)} = a_k^{(n)}-a_k^{(n-1)}-a_{k-2}^{(n-1)},
\end{equation}
with the understanding that $a_k^{(n-1)}=0$ whenever $k<0$ or $k>2n-2$. For even
$k$, we combine \eqref{2.13} with \eqref{4.9}, and after some tedious but
straightforward manipulations we get \eqref{4.7}. Similarly, combining
\eqref{2.14} with \eqref{4.9} gives \eqref{4.8}.
\end{proof}

\section{Resultants and discriminants}\label{sec:5}

The resultant of two polynomials
\begin{equation}\label{5.1}
\begin{cases}
&f(x)=a_\mu x^\mu+\cdots+a_1x+a_0,\\
&g(x)=b_m x^m+\cdots+b_1x+b_0,
\end{cases}
\end{equation}
of degrees $\mu$ and $m$, respectively, is an important object in algebra and
number theory which, roughly speaking, measures how far the zeros
$\alpha_1,\ldots, \alpha_\mu$ of $f$ are apart from the zeros
$\beta_1,\ldots, \beta_m$ of $g$. The resultant can be defined by
\[
{\rm Res}(f,g)=a_{\mu}^mb_m^{\mu}\prod_{i=1}^{\mu}
\prod_{j=1}^m\left(\alpha_i-\beta_j\right),
\]
which shows that it vanishes if and only if $f$ and $g$ have a 
zero in common. Another useful and equivalent definition is by way of the
Sylvester determinant, the determinant of a certain $(\mu+m)\times(\mu+m)$
matrix which has the coefficients of $f$ and $g$ as entries. This implies
that the resultant of two polynomials with integer coefficients is always an
integer. Further important properties are 
\begin{align}
{\rm Res}(f,g) &= a_{\mu}^m\prod_{j=1}^\mu g(\alpha_j),\label{5.3}\\
{\rm Res}(f,g) &= (-1)^{\mu m}{\rm Res}(g,f),\label{5.4}
\end{align}
and when $p$ and $q$ are arbitrary polynomials in $x$, then
\begin{equation}\label{5.5}
{\rm Res}(f,pq) = {\rm Res}(f,p)\cdot {\rm Res}(f,q).
\end{equation}

Finally, we require the following lemma, which can be found as Lemma~4.1 in
\cite{DS} or, with a different normalization, in \cite[p.~58]{PZ}. All other
properties mentioned above can also be found in \cite{DS}, with extensions and
references.

\begin{lemma}\label{lem:5.1}
Let $f, g$ be as in \eqref{5.1}. If we can write
\[
f(x) = q(x)g(x) + r(x)
\]
with polynomials $q, r$ and $\nu:=\deg{r}$, then
\begin{equation}\label{5.7}
{\rm Res}(g,f) = b_m^{\mu-\nu}{\rm Res}(g,r).
\end{equation}
\end{lemma}

We are now ready to state and prove evaluations of certain resultants 
involving the polynomials $p_n(z)$ and $q_n(z)$.

\begin{proposition}\label{prop:5.2}
For all $n\geq 0$ we have
\begin{equation}\label{5.8}
{\rm Res}(p_n, p_{n+1}) = (n+2)^{2n}\cdot 2^{4n(n+1)}.
\end{equation}
\end{proposition}

\begin{proof} 
We apply Lemma~5.1 to \eqref{2.7} by setting $f(x)=p_{n+1}(x)$, $g(x)=p_n(x)$,
$q(x)=(x-1)^2$, and $r(x)=(n+2)(x+1)^{2n+2}$. We also note that $\mu=\nu=2n+2$,
and so by \eqref{5.7} we have
\[
{\rm Res}(p_n(x), p_{n+1}(x)) = {\rm Res}(p_n(x), (n+2)(x+1)^{2n+2}).
\]
From \eqref{5.3} and \eqref{2.6} we obtain
\[
{\rm Res}((n+2)(x+1)^{2n+2},p_n(x))
=(n+2)^{2n}\prod_{j=1}^{2n+2}p_n(-1) = (n+2)^{2n}\left(2^{2n}\right)^{2n+2},
\]
and with \eqref{5.4} we obtain \eqref{5.8}. 
\end{proof}

For the sake of comparison with \eqref{5.8}, it is worth mentioning that
\[
{\rm Res}(U_n, U_{n+1})=(-1)^\frac{n(n+1)}{2} 2^{n(n+1)},\quad
{\rm Res}(T_n, T_{n+1})=(-1)^\frac{n(n+1)}{2} 2^{(n-1)n}.
\]
These and other more general results of this type can be found in \cite{JRT},
\cite{Lo}, and \cite{Ya}. Next we prove two identities involving both 
polynomial sequences $p_n(x)$ and $q_n(x)$.

\begin{proposition}\label{prop:5.3}
For all $n\geq 0$ we have
\begin{equation}\label{5.10}
{\rm Res}(p_n, q_{n+1}) = {\rm Res}(p_n, p_{n+1}) = (n+2)^{2n}\cdot 2^{4n(n+1)}
\end{equation}
and
\begin{equation}\label{5.11}
{\rm Res}(p_n, q_n) = (n+1)^{2n-2}\cdot 2^{2(n-1)(2n+1)}\cdot
\begin{cases}
(n+2)^2 & (n\; \hbox{even}),\\
(n+1)^2 & (n\; \hbox{odd}).
\end{cases}
\end{equation}
\end{proposition}

\begin{proof} 
To prove \eqref{5.10}, we rewrite \eqref{4.3} as
\[
p_{n+1}(x)=(1+x^2)p_n(x)+q_{n+1}(x).
\]
Since $\mu=\deg{p_{n+1}}=2n+2=\deg{q_{n+1}}=\nu$, Lemma~\ref{lem:5.1} gives
\[
{\rm Res}(p_n, p_{n+1}) = {\rm Res}(p_n, q_{n+1}),
\]
and with \eqref{5.8} we get the second equality in \eqref{5.10}.

Next, we write \eqref{4.3} as
\[
p_n(x)=q_n(x)+(1+x^2)p_{n-1}(x).
\]
With the aim of using Lemma~\ref{lem:5.1} again, we take $q(x)=1$ (identically)
and $r(x)=(1+x^2)p_{n-1}(x)$, so that we have $\mu=\deg{p_n}=2n=\deg{r}=\nu$.
Now \eqref{5.7} and \eqref{5.5} give
\begin{equation}\label{5.13}
{\rm Res}(p_n, q_n) = {\rm Res}(q_n, (1+x^2)p_{n-1})
={\rm Res}(q_n, 1+x^2)\cdot {\rm Res}(q_n, p_{n-1}).
\end{equation}
By \eqref{5.3} and \eqref{4.3} we have
\begin{equation}\label{5.14}
{\rm Res}(q_n, 1+x^2)=q_n(i)q_n(-i) = p_n(i)p_n(-i),
\end{equation}
and by \eqref{2.4}, 
\begin{equation}\label{5.15}
p_n(\pm i) = (\pm i-1)^{2n}\sum_{k=0}^n(k+1)
\left(\frac{\pm i+1}{\pm i-1}\right)^{2k}.
\end{equation}
Since $(\frac{\pm i+1}{\pm i-1})^2=-1$ and $(i-1)(-i-1)=2$, we get with
\eqref{5.15},
\begin{equation}\label{5.16}
p_n(i)p_n(-i) = 2^{2n}\left(\sum_{k=0}^n(-1)^k(k+1)\right)^2.
\end{equation}
Finally, it is easy to see that the sum in parentheses in \eqref{5.16} is
\[
1-2+3-4+\cdots+(-1)^n(n+1) = \begin{cases}
\tfrac{1}{2}(n+2) & (n\; \hbox{even}),\\
-\tfrac{1}{2}(n+1) & (n\; \hbox{odd}).
\end{cases}
\]
Combining this with \eqref{5.16}, \eqref{5.14}, \eqref{5.13} and \eqref{5.10},
we get \eqref{5.11}, as desired.
\end{proof}

One of the most important invariants of a polynomial is its discriminant. Since
we already noted certain similarities between our polynomials $p_n(x)$ and
$ q_n(x)$ and the Chebyshev polynomials, we mention the fairly recent papers
\cite{DS, DU, GI, St, Tr1, Tr2} which have dealt with discriminants of
Chebyshev-like polynomials.

We begin with recalling the definition of the discriminant. Given a polynomial
$f(x)$ as in \eqref{5.1}, the discriminant of $f$ is usually defined by
\begin{equation}\label{5.17}
{\rm Disc}(f)=(-1)^{\frac{\mu(\mu-1)}{2}}a_{\mu}^{-1}{\rm Res}(f,f'),
\end{equation}
where $f'$ is the derivative of $f$. With \eqref{5.3} we then have
\begin{equation}\label{5.18}
{\rm Disc}(f)=(-1)^{\frac{\mu(\mu-1)}{2}}a_{\mu}^{\mu-2}
\prod_{j=1}^{\mu}f'(\alpha_j),
\end{equation}
where $\alpha_1, \ldots, \alpha_{\mu}$ are the zeros of $f$. It follows from
\eqref{5.18} that ${\rm Disc}(f)=0$ if and only if $f$ has multiple zeros, 
which is a key property of the discriminant. We are now ready to prove the
following result.

\begin{proposition}\label{prop:5.4}
For all $n\geq 1$ we have
\begin{equation}\label{5.19}
{\rm Disc}(p_n) = (-1)^n 2^{4n^2+2}(n+1)^{2n-3}(n+2)^{2n-2}.
\end{equation}
\end{proposition}

For the proof of \eqref{5.19} we require the following derivative identity.

\begin{lemma}\label{lem:5.5}
For all $n\geq 1$ we have
\begin{equation}\label{5.20}
p_n'(z) = \frac{2n}{z-1}\cdot p_n(z)-4(z-1)^{2n-2}
\sum_{k=1}^n k(k+1)\left(\frac{z+1}{z-1}\right)^{2k-1}.
\end{equation}
\end{lemma}

\begin{proof}
We differentiate both sides of \eqref{2.4} and then use \eqref{2.4} again. After
some straightforward manipulations we then get \eqref{5.20}.
\end{proof}

\begin{proof}[Proof of Proposition~\ref{prop:5.4}]
We use \eqref{5.18} with $f=p_n, \mu=2n$, and $a_{\mu}=(n+1)(n+2)/2$. Then
\begin{equation}\label{5.21}
{\rm Disc}(p_n)=(-1)^n \left(\frac{(n+2)(n+1)}{2}\right)^{2n-2}
\cdot\prod_{i=1}^{2n}p_n'(\alpha_i),
\end{equation}
where $\alpha_1, \ldots, \alpha_{2n}$ are the zeros of $p_n(z)$. Now, by
\eqref{5.20} we have
\begin{equation}\label{5.22}
p_n'(\alpha_i) =-4\cdot \frac{(\alpha_i-1)^{2n-1}}{\alpha_i+1}\cdot 
\sum_{k=1}^n k(k+1)\left(\frac{\alpha_i+1}{\alpha_i-1}\right)^{2k}.
\end{equation}
To evaluate the sum in \eqref{5.22}, we begin with the identity
\[
1+y+y^2+\cdots+y^{n+1} = \frac{1-y^{n+2}}{1-y}.
\]
After differentiating twice, we get
\begin{equation}\label{5.23}
\sum_{k=1}^n k(k+1)y^{k-1} = \frac{1}{(1-y)^2}\left(2\cdot\frac{1-y^{n+1}}{1-y}
-(n+1)(n+2)y^n+n(n+1)y^{n+1}\right).
\end{equation}
Next, we set $z=\alpha_i$ in \eqref{2.10} for any $i=1,2, \ldots, 2n$. 
Writing $\alpha=\alpha_i$ for convenience of notation, we get
\begin{equation}\label{5.23a}
(\alpha-1)^2
=(\alpha^2-(4n+6)\alpha+1)\left(\frac{\alpha+1}{\alpha-1}\right)^{2(n+1)},
\end{equation}
which can be rewritten as
\[
(\alpha-1)^2-(\alpha-1)^2\left(\frac{\alpha+1}{\alpha-1}\right)^{2(n+1)}
=(n+1)\left((\alpha-1)^2-(\alpha+1)^2\right)
\left(\frac{\alpha+1}{\alpha-1}\right)^{2(n+1)},
\]
and after some further straightforward manipulation we get
\begin{equation}\label{5.24}
\frac{1-y^{n+1}}{1-y}=(n+1)y^{n+1},\qquad\hbox{where}\quad
y= \left(\frac{\alpha+1}{\alpha-1}\right)^2.
\end{equation}
Substituting this into the right-hand side of \eqref{5.23}, we get
\begin{align}
\sum_{k=1}^n k(k+1)y^{k-1} &= \frac{1}{(1-y)^2}\left(2(n+1)y^{n+1}
-(n+1)(n+2)y^n+n(n+1)y^{n+1}\right) \label{5.25}\\
&=-\frac{(n+1)(n+2)}{1-y}y^{n},\nonumber
\end{align}
where $y$ is as given in \eqref{5.24}. By \eqref{5.23a}, an easy calculation 
gives
\begin{equation}\label{5.26}
\frac{1}{1-\left(\frac{\alpha+1}{\alpha-1}\right)^2}\cdot
\left(\frac{\alpha+1}{\alpha-1}\right)^{2n+2}
=\frac{(\alpha+1)^{2n+2}}{-4\alpha(\alpha-1)^{2n}}.
\end{equation}
With \eqref{5.22}, \eqref{5.25} and \eqref{5.26} we then get
\[
p_n'(\alpha_i) 
=-(n+1)(n+2)\cdot\frac{(\alpha_i+1)^{2n+1}}{\alpha_i(\alpha_i-1)},
\]
and with \eqref{5.21},
\begin{equation}\label{5.28}
{\rm Disc}(p_n)=(-1)^n \frac{((n+1)(n+2))^{4n-2}}{2^{2n-2}}
\cdot\prod_{i=1}^{2n}\frac{(\alpha_i+1)^{2n+1}}{\alpha_i(\alpha_i-1)}.
\end{equation}
To evaluate the product in \eqref{5.28}, we use the fact that the identity
\[
p_n(z)=\frac{(n+1)(n+2)}{2}\cdot\prod_{i=1}^{2n}(z-\alpha_i),
\]
together with \eqref{2.6}, implies
\[
\prod_{i=1}^{2n}(1-\alpha_i)=\frac{2^{2n+1}}{n+2},\qquad
\prod_{i=1}^{2n}(1+\alpha_i)=\frac{2^{2n+1}}{(n+1)(n+2)},\qquad
\prod_{i=1}^{2n}\alpha_i=1.
\]
Finally, these identities, together with \eqref{5.28}, give \eqref{5.19} after
collecting the powers of 2, $n+1$ and $n+2$.
\end{proof}

No other identities such as those in Propositions~\ref{prop:5.2}--\ref{prop:5.4}
seem to exist, and in particular there seems to be no such
identity (that is, having only small prime factors) for Disc$(q_n)$.

\section{Some irreducibility results}\label{sec:6}

Computations using the {\tt irreduc} function in Maple suggest that the 
polynomials $p_n(x)$ and $q_n(x)$ are irreducible over $\mathbb Q$. We are 
going to prove the following special case of $p_n(x)$.

\begin{theorem}\label{thm:6.1}
The polynomials $p_n(x)$ are irreducible over $\mathbb Q$ when $n=p-2$, 
where $p$ is an odd prime.
\end{theorem}

We begin by recalling that $p_n(x)$ (and also $q_n(x)$) are 
reciprocal polynomials with real coefficients, that is,
\[
x^{2n}p_n(\tfrac{1}{x})=p_n(x).
\]
This fact will play an important role in the proof of Theorem~\ref{thm:6.1}
since we can use an interesting recent irreducibility criterion.

In what follows, we use the terminology of Cafure and Cesaretto \cite{CC},
adapting their notation to be consistent with ours. It is clear that a 
reciprocal polynomial $p\in{\mathbb Q}[x]$ of degree $2n$ can be written as
\begin{equation}\label{6.2}
p(x)=a_0x^n+\sum_{k=1}^n a_k\left(x^{n+k}+x^{n-k}\right)
=x^n\left(a_0+\sum_{k=1}^n a_k\left(x^k+x^{-k}\right)\right).
\end{equation}
Since we can write each of the terms $x^k+x^{-k}$ as a polynomial in
$x+x^{-1}$, there is a unique polynomial $f\in{\mathbb Q}[x]$ satisfying the
equation
\begin{equation}\label{6.3}
p(x) = x^nf(x+\tfrac{1}{x}).
\end{equation}
Cafure and Cesaretto \cite{CC} call this polynomial $f$ the image polynomial
under the mapping ${\mathsf R}:p\mapsto f$ described above. We are going to use
two results from \cite{CC}, slightly adapted to our situation.

\begin{lemma}[\cite{CC}, p.~42]\label{lem:6.2}
Let $p\in{\mathbb Q}[x]$ be a reciprocal polynomial of degree $2n$. Then the
image polynomial $f\in{\mathbb Q}[x]$, defined by \eqref{6.3}, is given by
\begin{equation}\label{6.4}
f(x)=a_0+2\sum_{k=1}^n a_kT_k(\tfrac{x}{2}),
\end{equation}
where $a_0, a_1,\ldots, a_n$ are as in \eqref{6.2}, and $T_k(y)$ is the
Chebyshev polynomial defined in \eqref{3.12}.
\end{lemma}

Before stating the next result, we recall that the {\it content\/} of a 
polynomial $p\in{\mathbb Z}[x]$, denoted by ${\rm cont}(p)$, is the greatest
common divisor of the coefficients of $p$. If ${\rm cont}(p)=1$, then $p$ is
called {\it primitive\/}.

\begin{theorem}[\cite{CC}, Theorem~11(1)]\label{thm:6.3}
Let $p\in{\mathbb Z}[x]$ be a primitive reciprocal polynomial of even degree.
If the image polynomial $f(x)$, defined by \eqref{6.3}, is irreducible and if
$|p(1)|$ or $|p(-1)|$ are not perfect squares, then $p$ is irreducible over
$\mathbb Q$.
\end{theorem}

In order to apply Theorem~\ref{thm:6.3}, we first need a result on the
content of the polynomials $p_n(x)$. 

\begin{lemma}\label{lem:6.4}
Let $n\geq 1$ be an odd integer, and write $n+1=2^km$, where $k\geq 1$ and $m$ 
is odd. Then ${\rm cont}(p_n)=2^{k-1}$.
\end{lemma}

As an important tool for the proof of this lemma, we use a well-known 
congruence of E.~Lucas (1878) in two different forms; see, e.g., \cite{Gr}.
Let $p$ be a prime and $n,m\in{\mathbb N}$, written in base $p$ as
\begin{align*}
n &= n_dp^d+n_{d-1}p^{d-1}+\dots+n_1p+n_0,\quad 0\leq n_j<p,\\
m &= m_dp^d+m_{d-1}p^{d-1}+\dots+m_1p+m_0,\quad 0\leq m_j<p.
\end{align*}
Then
\begin{equation}\label{6.5}
\binom{n}{m}\equiv \binom{\lfloor n/p\rfloor}{\lfloor m/p\rfloor}
\binom{n_0}{m_0}\pmod{p}
\end{equation}
and 
\begin{equation}\label{6.6}
\binom{n}{m}\equiv\binom{n_d}{m_d}\binom{n_{d-1}}{m_{d-1}}\dots
\binom{n_0}{m_0}\pmod{p}
\end{equation}
with the usual convention that $\binom{r}{s}=0$ when $r<s$.

\begin{proof}[Proof of Lemma~\ref{lem:6.4}]
Using \eqref{2.13} and an elementary identity for binomial coefficients, we
see that
\begin{equation}\label{6.6a}
a_{2j}^{(n)}=\frac{n+2}{4}\binom{2n+2}{2j+1}
=\frac{(n+2)(n+1)}{2(2j+1)}\binom{2n+1}{2j}\qquad(0\leq j\leq n)
\end{equation}
are integers divisible by $2^{k-1}$ since $n+1=2^km$, with $m$ odd.
For the same reason, by \eqref{2.14} we also have $2^{k-1}|a_{2j+1}^{(n)}$,
$0\leq j\leq n-1$.

On the other hand, by \eqref{2.13} we have $a_0^{(n)}=(n+1)(n+2)/2$, so 
$2^{k-1}$ divides $a_0^{(n)}$ exactly, which means that 
$2^{k-1}\|{\rm cont}(p_n)$.

Next, suppose that an odd prime $p$ is such that $p^\alpha\|m$, and thus
$p^\alpha\|n+1$ for some $\alpha\geq 1$. Now let $j$ be such that 
$2j+1=p^\alpha$; we are going to show that
\begin{equation}\label{6.7}
p\nmid a_{2j}^{(n)},\quad\hbox{and thus}\quad {\rm cont}(p_n)=2^{k-1}.
\end{equation}
To prove \eqref{6.7}, we use \eqref{6.6a} and first note that the powers of $p$
in $n+1$ and $2j+1$ cancel, while $p\nmid n+2$. Therefore it remains to show
that $p\nmid\binom{2n+1}{2j}$.

Since by our assumptions on $n$ and $j$ we have $2n+1\equiv 2j\equiv-1\pmod{p}$,
we can apply \eqref{6.5} with $n_0=m_0=p-1$. On the other
hand, we have
\[
2n+1-(p-1) = 2(n+1)-p = 2Bp^\alpha-p
\]
for some $B\in{\mathbb N}$, so that
\begin{equation}\label{6.8}
\left\lfloor\frac{2n+1}{p}\right\rfloor = 2Bp^{\alpha-1}-1,
\end{equation}
and $2j-(p-1)=2j+1-p=p^\alpha-p$, so that
\begin{equation}\label{6.9}
\left\lfloor\frac{2j}{p}\right\rfloor = p^{\alpha-1}-1.
\end{equation}
But we can also write
\[
2Bp^{\alpha-1}-1 = (2B-1)p^{\alpha-1}+(p^{\alpha-1}-1),
\]
so that by \eqref{6.6},
\[
\binom{2Bp^{\alpha-1}-1}{p^{\alpha-1}-1} \equiv
\binom{2B-1}{0}\binom{p^{\alpha-1}-1}{p^{\alpha-1}-1}\equiv 1\cdot 1\pmod{p}.
\]
Altogether, we therefore have with \eqref{6.5}, \eqref{6.8} and \eqref{6.9}
that $p\nmid\binom{2n+1}{2j}$, which proves \eqref{6.7} and thus the lemma.
\end{proof}

\noindent
{\bf Remark.} Although this is not needed here, we can also show that, with 
the notation of Lemma~\ref{lem:6.4}, cont$(p_{n-1})=2^{k-1}$ when $n\geq 1$ is
odd, and cont$(q_n)=2^k$ for all $n\geq 0$. The proofs are similar to that of
Lemma~\ref{lem:6.4}.

\medskip
In view of Lemma~\ref{lem:6.2} and Theorem~\ref{thm:6.3}, we let $p(x)=p_n(x)$ 
and denote the corresponding image polynomial by $f(x)=f_n(x)$. These last
polynomials have the following explicit expansions for odd $n$.

\begin{lemma}\label{lem:6.6}
For any integer $m\geq 0$, we have
\begin{align}
f_{2m+1}(x)&=\sum_{j=0}^m 2^{2m-2j+1}(m+1-j)\binom{2m+3}{2j}x^{2j}\label{6.19}\\
&+\sum_{j=1}^{m+1} 2^{2m-2j+2}\frac{j(2j+1)}{2m+3-2j}\binom{2m+3}{2j+1}x^{2j-1}.\nonumber
\end{align}
\end{lemma}

\medskip
For the proof of this lemma we require the following binomial identities.

\begin{lemma}\label{lem:6.5}
For integers $1\leq j\leq m$ we have
\begin{align}
\sum_{\ell=j}^m(-1)^{\ell-j}&\frac{(m-\ell+1)\ell}{(2m+2\ell+3)(\ell+j)}
\binom{4m+4}{2m+2\ell+1}\binom{\ell+j}{\ell-j}\label{6.10} \\
&\qquad\qquad=2^{2m-2j}\frac{m+1-j}{j}\binom{2m+2}{2j-1},\nonumber \\
\sum_{\ell=j}^{m+1}(-1)^{\ell-j}&\frac{2\ell-1}{\ell+j-1}
\binom{4m+4}{2m+2\ell+1}\binom{\ell+j-1}{\ell-j} \label{6.11}\\
&\qquad\qquad=2^{2m-2j+4}\frac{j}{2m+3-2j}\binom{2m+2}{2j}.\nonumber
\end{align}
\end{lemma}

\begin{proof}(Outline).
We first conjectured the right-hand sides of both identities by numerical
experimentations with some fixed small values of $j\geq 1$. Then we used the
WZ (Wilf-Zeilberger) method, as explained, e.g., in \cite{PWZ}, to prove these 
identities. This can be done with Maple, using the function {\tt WZMethod}
within the package {\tt SumTools[Hypergeometric]}, or with Mathematica, using
the function {\tt Zb} in the package {\tt fastZeil}. We leave the details to the interested reader.
\end{proof}

\begin{proof}[Proof of Lemma~\ref{lem:6.6}]
In view of Lemma~\ref{lem:6.2}, we use the following well-known explicit 
formula for the Chebyshev polynomials $T_k(x)$, namely
\begin{equation}\label{6.20}
2T_k(\tfrac{x}{2}) = k\sum_{j=0}^{\lfloor k/2\rfloor}\frac{(-1)^j}{k-j}
\binom{k-j}{j}x^{k-2j};
\end{equation}
see, e.g., \cite[eq.~(1.96)]{Ri}. Now, by \eqref{6.4} and with the notation of
\eqref{2.12}, we have with $n=2m+1$, 
\begin{equation}\label{6.21}
f_{2m+1}(x)=a_{2m+1}^{(2m+1)}
+2\sum_{\ell=1}^m a_{2m+2\ell+1}^{(2m+1)}T_{2\ell}(\tfrac{x}{2})
+2\sum_{\ell=1}^{m+1}a_{2m+2\ell}^{(2m+1)}T_{2\ell-1}(\tfrac{x}{2}).
\end{equation}
With \eqref{2.13}, \eqref{2.14} and \eqref{6.20} we then get
\begin{equation}\label{6.22}
f_{2m+1}(x)=\frac{m+1}{2}\binom{4m+4}{2m+1}+S_1^{(m)}(x)+S_2^{(m)}(x),
\end{equation} 
where
\begin{align}
S_1^{(m)}(x)=\sum_{\ell=1}^m&\frac{(2m+3)(m-\ell+1)}{2m+2\ell+3}
\binom{4m+4}{2m+2\ell+1}\ell \label{6.23}\\
&\qquad\qquad\times\sum_{j=0}^{\ell}\frac{(-1)^j}{2\ell-j}
\binom{2\ell-j}{j}x^{2\ell-2j},\nonumber\\
S_2^{(m)}(x)=\sum_{\ell=1}^{m+1}&\frac{2m+3}{4}
\binom{4m+4}{2m+2\ell+1}(2\ell-1) \label{6.24}\\
&\qquad\qquad\times\sum_{j=0}^{\ell-1}\frac{(-1)^j}{2\ell-1-j}
\binom{2\ell-1-j}{j}x^{2\ell-1-2j}.\nonumber
\end{align}
To simplify \eqref{6.22}, we first rewrite the inner sums in \eqref{6.23} and
\eqref{6.24} respectively as
\begin{equation}\label{6.25}
\sum_{j=0}^{\ell}\frac{(-1)^{\ell-j}}{\ell+j}\binom{\ell+j}{\ell-j}x^{2j},\qquad
\sum_{j=1}^{\ell}\frac{(-1)^{\ell-j}}{\ell+j-1}\binom{\ell+j-1}{\ell-j}x^{2j-1}.
\end{equation}
Separating out the constant coefficient in \eqref{6.23} and changing the order
of summation in the remaining double sum, we get
\begin{align}
&S_1^{(m)}=\sum_{\ell=1}^m(-1)^\ell\frac{(2m+3)(m-\ell+1)}{2m+2\ell+3}
\binom{4m+4}{2m+2\ell+1}\label{6.26}\\
&+(2m+3)\sum_{j=1}^m\left(
\sum_{\ell=j}^m\frac{(-1)^{\ell-j}(m-\ell+1)\ell}{(2m+2\ell+3)(\ell+j)}
\binom{4m+4}{2m+2\ell+1}\binom{\ell+j}{\ell-j}\right)x^{2j}.\nonumber
\end{align}
Similarly, using the second sum in \eqref{6.25} and changing the order of
summation in \eqref{6.24}, we get
\begin{equation}\label{6.27}
S_2^{(m)}=\frac{2m+3}{4}\sum_{j=1}^{m+1}\left(
\sum_{\ell=j}^{m+1}\frac{(-1)^{\ell-j}(2\ell-1)}{\ell+j-1}
\binom{4m+4}{2m+2\ell+1}\binom{\ell+j-1}{\ell-j}\right)x^{2j-1}.
\end{equation}
The first sum in \eqref{6.26}, which we denote by $S_3^{(m)}$, can be evaluated 
symbolically with a computer algebra system, and after some simplification 
we can write it as
\begin{equation}\label{6.28}
S_3^{(m)}=(m+1)2^{2m+1}-\frac{m+1}{2}\binom{4m+4}{2m+1}.
\end{equation}
We now substitute \eqref{6.28} and \eqref{6.10} into \eqref{6.26}, as well as
\eqref{6.11} into \eqref{6.27}. Then after some simplifications, \eqref{6.22}
yields
\begin{align*}
f_{2m+1}(x)&=(m+1)2^{2m+1}
+\sum_{j=1}^m 2^{2m-2j+1}(m+1-j)\binom{2m+3}{2j}x^{2j}\\
&+\sum_{j=1}^{m+1} 2^{2m-2j+2}\frac{j(2j+1)}{2m+3-2j}\binom{2m+3}{2j+1}x^{2j-1}.
\end{align*}
Finally, we note that the term $(m+1)2^{2m+1}$ is just the case $j=0$ in the
sum following this term. This proves the desired identity \eqref{6.19}.
\end{proof}

The first few polynomials $f_{2m+1}(x)$ are shown in Table~3. By evaluating
\[
x^{2m+1}f_{2m+1}(x+\tfrac{1}{x}),\quad m=1, 2, 3,
\]
we obtain $p_1(x), p_3(x)$, and $p_5(x)$, respectively, as shown in Table~1.
This illustrates the key equation \eqref{6.3} numerically.

We can see from the coefficients in Table~3 that $f_n(x)$, for 
$n=1, 3, 5$, and 9 are $p$-Eisenstein polynomials with $p=3, 5, 7$, and 11,
respectively. This is true in general:

\begin{theorem}\label{thm:6.7}
For any odd prime $p$, the polynomial $f_{p-2}(x)$ is $p$-Eisenstein, and is
therefore irreducible over $\mathbb Q$.
\end{theorem}

\begin{proof}
We use \eqref{6.19} with $2m+1=p-2$, that is, $2m+3=p$. It is a well-known
and easy fact that $p\mid\binom{p}{\ell}$ for $1\leq\ell\leq p-1$. Hence it 
suffices to consider the constant and the leading coefficients of $f_{2m+1}(x)$.
The former is $(m+1)2^{2m+1}$, which is obviously not divisible by $p$. The 
latter occurs for $j=m+1$ in the second sum and is $(m+1)(2m+3)=(p-1)p/2$.
This is divisible by $p$ but not by $p^2$, which is the last requirement for
$f_{p-2}(x)$ to be an Eisenstein polynomial.
\end{proof}

\medskip
\begin{center}
{\renewcommand{\arraystretch}{1.1}
\begin{tabular}{|r|l|}
\hline
$n$ & $f_n(z)$ \\
\hline
1 & $3x+2$ \\
3 & $10x^3 + 20x^2 + 40x + 16$ \\
5 & $21x^5 + 70x^4 + 280x^3 + 336x^2 + 336x + 96$  \\
7 & $36x^7 + 168x^6 + 1008x^5 + 2016x^4 + 4032x^3 + 3456x^2 + 2304x + 512$  \\
9 & $55x^9 + 330x^8 + 2640x^7 + 7392x^6 + 22176x^5 + 31680x^4 + 42240x^3$ \\
& $\quad+ 28160x^2 + 14080x + 2560$  \\
\hline
\end{tabular}}

\medskip
{\bf Table~3}: $f_{2m+1}(z)$ for $0\leq m\leq 4$.
\end{center}


\begin{proof}[Proof of Theorem~\ref{thm:6.1}]
Obviously $n=p-2$ is odd, and we first consider $n=4\ell+1$. By
Lemma~\ref{lem:6.4} we have cont$(p_n)=1$ since $k=1$. By \eqref{2.6} we
have $p_{4\ell+1}(1)=(4\ell+2) 4^{4\ell+1}$, which is never a square since
$4\ell+2$ cannot be a square.

When $n=4\ell+3$, then with $k$ defined by $(4\ell+4)=2^km$, where $m$ is odd,
Lemma~\ref{lem:6.4} tells us that $\widetilde{p}_n(x):=2^{1-k}p_n(x)$ is
primitive. Next, again by \eqref{2.6}, we have
\[
\widetilde{p}_{4\ell+3}(1)=2^{1-k}(4\ell+4)2^{8\ell+6}=m\cdot 2^{8\ell+7},
\]
where we have used the definition of $k$. Since this last expression is never
a square, two of the conditions of Theorem~\ref{thm:6.3} are satisfied for
both $p_{4\ell+1}(x)$ and $\widetilde{p}_{4\ell+3}(x)$. The third and most
important condition is satisfied by Theorem~\ref{thm:6.7}, which completes the
proof of Theorem~\ref{thm:6.1}.
\end{proof}

\section{Some factorization results}

In contrast to the irreducibility (or conjectured irreducibility) of $p_n(z)$
and $q_n(z)$, we will now see that the polynomials consisting of only the
even-power terms of $p_n(z)$ factor in a specific way. The same is true for
the polynomials consisting of only the odd-power terms of $q_n(z)$. We begin
by defining the polynomials in question, namely
\begin{align*}
p_n^e(z)&=\frac{p_n(z)+p_n(-z)}{2},\qquad p_n^o(z)=\frac{p_n(z)-p_n(-z)}{2},\\
q_n^e(z)&=\frac{q_n(z)+q_n(-z)}{2},\qquad q_n^o(z)=\frac{q_n(z)-q_n(-z)}{2}.
\end{align*}
Two of these are closely related to each other, as we will now see.

\begin{lemma}\label{lem:6.1a}
For $n\geq 0$ we have
\begin{equation}\label{6.1a}
q_{n+1}^o(z)=2z\,p_n^e(z).
\end{equation}
\end{lemma}

\begin{proof}
Using the definition and \eqref{4.3a}, we have
\[
q_{n+1}^o(z)=\frac{n+2}{2}\left((1+z)^{2n+2}-(1-z)^{2n+2}\right)
-z\left(p_n(z)+p_n(-z)\right).
\]
Subtracting $2zp_n^e(z)$ from both sides, we get
\[
q_{n+1}^o(z)-2zp_n^e(z)=\frac{n+2}{2}\left((1+z)^{2n+2}-(1-z)^{2n+2}\right)
-2z\left(p_n(z)+p_n(-z)\right).
\]
By \eqref{3.7}, the right-hand side vanishes, which proves \eqref{6.1a}.
\end{proof}

By Lemma~\ref{lem:6.1a}, we can restrict our attention to $p_n^e(z)$. We have
the following factorization.

\begin{proposition}\label{prop:6.2a}
For all $n\geq 0$ we have
\begin{equation}\label{6.2a}
p_n^e(z)=\frac{n+2}{2}
\left((1-z^2)^{\frac{n+1}{2}}T_{n+1}\big(\tfrac{1}{\sqrt{1-z^2}}\big)\right)
\left((1-z^2)^{\frac{n}{2}}U_{n}\big(\tfrac{1}{\sqrt{1-z^2}}\big)\right).
\end{equation}
\end{proposition}

\begin{proof}
We use \eqref{3.13} with $z=\sqrt{x^2-1}/x$. Solving for $x$, we get
$x=\pm 1/\sqrt{1-z^2}$. Since the left-hand side of \eqref{3.13} is an even
function in $x$, we may choose the ``$+$" sign. Upon dividing \eqref{3.13} by
2, we then get
\[
p_n^e(z)=\frac{n+2}{4}
(1-z^2)^{\frac{2n+1}{2}}U_{2n+1}\big(\tfrac{1}{\sqrt{1-z^2}}\big).
\]
Finally, by using the identity $U_{2n+1}(x)=2T_{n+1}(x)U_{n}(x)$ (see, e.g., 
\cite[p.~229]{Ri}) and appropriately distributing the powers of $\sqrt{1-z^2}$,
we get \eqref{6.2a}.
\end{proof}

The following is a consequence of Proposition~\ref{prop:6.2a}.

\begin{corollary}\label{cor:6.3a}
For $n\geq 1$, we have $p_n^e(z)=r_n(z)s_n(z)$, where both factors are even
polynomials with integer coefficients, and
\[
\begin{cases}
\deg{r_n}=\deg{s_n}=n & \hbox{when $n$ is even},\\
\deg{r_n}=n+1, \deg{s_n}=n-1 & \hbox{when $n$ is odd}.
\end{cases}
\]
\end{corollary}

\begin{proof}
The Chebyshev polynomials $T_n(x)$ and $U_n(x)$ have degree $n$ and they are
even (resp.\ odd) polynomials when $n$ is even (resp.\ odd). This implies that
$x^nT_n(\tfrac{1}{x})$ and $x^nU_n(\tfrac{1}{x})$ are even polynomials for all 
$n$, and they are
of degree $n$ (resp.\ $n-1$) when $n$ is even (resp.\ odd). The same is then
true when $x$ is replaced by $\sqrt{1-z^2}$. The statement of the corollary
now follows from \eqref{6.2a}, and Gauss's Lemma shows that both factors have
integer coefficients.
\end{proof}

The polynomials $r_n(z)$ and $s_n(z)$ may or may not be irreducible over 
$\mathbb Q$. In particular, in the cases where $T_{n+1}(x)$ or $U_n(x)$ have 
even or odd factors, $r_n(z)$ or $s_n(z)$ will also split into further factors.
We did not pursue this further, and only note that factorization properties of 
Chebyshev polynomials can be found in \cite[p.~220ff.]{Ri}.

\section*{Acknowledgment}

We thank Lin Jiu of Duke Kunshan University and Dalhousie University for his 
help and advise with the WZ method and its implementations.

\end{document}